\documentclass[twoside,leqno,10pt, A4]{amsart}
\DeclareMathAlphabet{\mathpzc}{OT1}{pzc}{m}{it}
\usepackage{amsfonts}
\usepackage{amsmath}
\usepackage{amscd}
\usepackage{amssymb}
\usepackage{amsthm}
\usepackage{amsrefs}
\usepackage{latexsym}
\usepackage{mathrsfs}
\usepackage{bbm}
\usepackage{enumerate}
\usepackage[mathscr]{eucal}
\usepackage{color}
\begin{document}

\newtheorem{theorem}[subsection]{Theorem}
\newtheorem{proposition}[subsection]{Proposition}
\newtheorem{lemma}[subsection]{Lemma}
\newtheorem{corollary}[subsection]{Corollary}
\newtheorem{conjecture}[subsection]{Conjecture}
\newtheorem{prop}[subsection]{Proposition}
\numberwithin{equation}{section}
\newcommand{\mr}{\ensuremath{\mathbb R}}
\newcommand{\mc}{\ensuremath{\mathbb C}}
\newcommand{\dif}{\mathrm{d}}
\newcommand{\intz}{\mathbb{Z}}
\newcommand{\ratq}{\mathbb{Q}}
\newcommand{\natn}{\mathbb{N}}
\newcommand{\comc}{\mathbb{C}}
\newcommand{\rear}{\mathbb{R}}
\newcommand{\prip}{\mathbb{P}}
\newcommand{\uph}{\mathbb{H}}
\newcommand{\fief}{\mathbb{F}}
\newcommand{\majorarc}{\mathfrak{M}}
\newcommand{\minorarc}{\mathfrak{m}}
\newcommand{\sings}{\mathfrak{S}}
\newcommand{\fA}{\ensuremath{\mathfrak A}}
\newcommand{\mn}{\ensuremath{\mathbb N}}
\newcommand{\mq}{\ensuremath{\mathbb Q}}
\newcommand{\half}{\tfrac{1}{2}}
\newcommand{\f}{f\times \chi}
\newcommand{\summ}{\mathop{{\sum}^{\star}}}
\newcommand{\chiq}{\chi \bmod q}
\newcommand{\chidb}{\chi \bmod db}
\newcommand{\chid}{\chi \bmod d}
\newcommand{\sym}{\text{sym}^2}
\newcommand{\hhalf}{\tfrac{1}{2}}
\newcommand{\sumstar}{\sideset{}{^*}\sum}
\newcommand{\sumprime}{\sideset{}{'}\sum}
\newcommand{\sumprimeprime}{\sideset{}{''}\sum}
\newcommand{\sumflat}{\sideset{}{^{\flat}}\sum}
\newcommand{\sumSTAR}{\sideset{}{^{\star}}\sum}
\newcommand{\shortmod}{\ensuremath{\negthickspace \negthickspace \negthickspace \pmod}}
\newcommand{\V}{V\left(\frac{nm}{q^2}\right)}
\newcommand{\sumi}{\mathop{{\sum}^{\dagger}}}
\newcommand{\mz}{\ensuremath{\mathbb Z}}
\newcommand{\leg}[2]{\left(\frac{#1}{#2}\right)}
\newcommand{\muK}{\mu_{\omega}}

\newcommand{\RR}{\mathbb{R}}
\newcommand{\QQ}{\mathbb{Q}}
\newcommand{\CC}{\mathbb{C}}
\newcommand{\NN}{\mathbb{N}}
\newcommand{\ZZ}{\mathbb{Z}}
\newcommand{\FF}{\mathbb{F}}
\newcommand{\C}{{\mathcal{C}}}
\newcommand{\OO}{{\mathcal{O}}}
\newcommand{\cc}{{\mathfrak{c}}}
\newcommand{\norm}{{\mathpzc{N}}}
\newcommand{\trace}{{\mathrm{Tr}}}
\newcommand{\ringO}{{\mathfrak{O}}}
\newcommand{\fa}{{\mathfrak{a}}}
\newcommand{\fb}{{\mathfrak{b}}}
\newcommand{\fc}{{\mathfrak{c}}}
\newcommand{\res}{{\mathrm{res}}}
\newcommand{\fp}{{\mathfrak{p}}}
\newcommand{\fm}{{\mathfrak{m}}}
\newcommand{\aut}{\rm Aut}
\newcommand{\mt}{m(t,u;\ell^k)}
\newcommand{\mtone}{m(t_1,u;\ell^k)}
\newcommand{\mttwo}{m(t_2,u;\ell^k)}
\newcommand{\mbadu}{m(t,u_0;\ell^k)}
\newcommand{\Sts}{S(t_1,t_2;\ell^k)}
\newcommand{\Stt}{S(t;\ell^k)}
\newcommand{\St}{R(t;\ell^k)}
\newcommand{\nN}{n(N,u;\ell^k)}
\newcommand{\Tnn}{T(N;\ell^k)}
\newcommand{\Tn}{T(N;\ell^k)}
\newcommand{\tilnul}{{\tilde{\nu}}_\ell}
\newcommand{\tilnulk}{{\tilde{\nu}}_\ell^{(k)}}
\makeatletter
\def\imod#1{\allowbreak\mkern7mu({\operator@font mod}\,\,#1)}
\makeatother

\title[Value-distribution of quartic Hecke $L$-functions]{Value-distribution of quartic Hecke $L$-functions}

\date{\today}
\author{Peng Gao and Liangyi Zhao}

\begin{abstract}
Set $K=\QQ(i)$ and suppose that $c\in \ZZ[i]$ is a square-free algebraic integer with $c\equiv 1 \imod{\langle16\rangle}$.  Let $L(s,\chi_{c})$ denote the Hecke $L$-function associated with the quartic residue character modulo $c$.  For $\sigma>1/2$, we prove an asymptotic distribution function $F_{\sigma}$ for the values of the logarithm of
\begin{equation*}
L_c(s)= L(s,\chi_c)L(s,\overline{\chi}_{c}),
\end{equation*}
  as $c$ varies.  Moreover, the characteristic function of $F_{\sigma}$ is expressed explicitly as a product over the prime ideals of $\ZZ[i]$.
\end{abstract}

\maketitle

\noindent {\bf Mathematics Subject Classification (2010)}: 11M41, 11R42  \newline

\noindent {\bf Keywords}: value-distribution, logarithm of $L$-functions, quartic characters

\section{Introduction}
\label{sec1}

Let $d$ be a non-square integer such that $d\equiv 0,1\imod 4$ and $\chi_d=\left(\frac{d}{.}\right)$ be the Kronecker symbol. In the early 1950s, S. Chowla and P. Erd\H{o}s studied the distribution of values of quadratic Dirichlet $L$-functions $L(s, \chi_d)$.
They proved in \cite{chowla-erdos} that when $\sigma>3/4$, then
$$\lim_{x\rightarrow \infty} \frac{\#\{0<d\leq x;~d\equiv 0, 1\imod{4}~{\rm and}~ L(\sigma, \chi_d)\leq z\}}{x/2}=G(z)$$
exists and the distribution function $G(z)$ is continuous and strictly increasing satisfying $G(0)=0$, $G(\infty)=1$. This result was further strengthened by P. D. T. A. Elliott for $\sigma=1$ in \cite{elliott-0}. \newline

   A systematic study of the value-distribution of the logarithm and the logarithmic derivative of $L$-functions on the half-plane $\Re(s)>1/2$ has been carried out by Y. Ihara and K. Matsumoto (see for example \cite{I-M1} and \cite{I-M}). Based on the approach in \cite{I-M}, M. Mourtada and V. K. Murty proved (\cite[Theorem 2]{M-M}), assuming the Generalized Riemann Hypothesis (GRH) for $L(s, \chi_d)$, that for any $\sigma>1/2$, there exists a probability density function $Q_\sigma$ such that
\begin{equation*}
\lim _{Y\rightarrow\infty} \frac{1}{\#\mathcal{F}(Y)} \# \left\{ d\in \mathcal{F}(Y), \frac {L^{\prime}(\sigma, \chi_d)}{L(\sigma, \chi_d)} \leq z \right\}= \int\limits_{-\infty}^{z} Q_\sigma(t) \dif t.
\end{equation*}
Here $\mathcal{F}(Y)$ denotes the set of the fundamental discriminants in the interval $[-Y, Y]$ \newline

If $d$ is a fundamental discriminant, $L(s, \chi_d)=\zeta_{\QQ(\sqrt{d})}(s)/\zeta(s)$, with $\zeta_{\QQ(\sqrt{d})}(s)$ denoting the Dedekind
zeta function of $\QQ(\sqrt{d})$ and $\zeta(s)$ the Riemann zeta function.  A. Akbary and A. Hamieh studied an analogue case of the above result of  Mourtada and Murty. Let $F=\QQ(\zeta_3)$ and $\ringO_F=\ZZ[\zeta_3]$ be the ring of integers of $F$, where $\zeta_3=\exp(2\pi i/3)$. Let $\mathcal{D}(Y)$ denote the set of square-free elements $d$ such that $d \equiv 1 \pmod 9$ and $\norm(d) \leq Y$.  Further define $L_{d, F}(s)=\zeta_{F(d^{1/3})}(s)/\zeta(s)$, where $\zeta_{F(d^{1/3})}(s)$ is the Dedekind zeta function of $F(d^{1/3})$.  Then Akbary and Hamieh \cite[Theorem 1.4]{AH} proved that, without assuming GRH, for either $\mathcal{L}_{d}(s)=\log L_d(s)$ or $L'_d/L_d(s)$, there exists a corresponding probability density function $D_\sigma$ such that for every $\sigma>1/2$,
\begin{equation*}
\lim _{Y\rightarrow\infty} \frac{1}{\#\mathcal{D}(Y)} \# \left\{ d \in \mathcal{D}(Y),  \mathcal{L}_{d}(\sigma) \leq z \right\}= \int\limits_{-\infty}^{z} D_{\sigma}(t) \dif t.
\end{equation*}

  We note that $L_{d, F}(s)$ can be decomposed as a product of Hecke $L$-functions. In fact, it is shown in the paragraph below \cite[(2)]{AH} that
\begin{equation}
\label{Ld}
L_{d, F}(s)= L(s,\chi_d)L(s,\overline{\chi}_{d}),
\end{equation}
   where $L(s,\chi_{d})$ is the Hecke $L$-function associated with the cubic residue symbol $\chi_{d}=\left(\frac{\cdot}{d}\right)_3$. \newline

Motivated by the above result, we consider the value-distribution of the logarithm of the product of quartic Hecke $L$-functions in this paper.  Set $K=\QQ(i)$ and $\mathcal{O}_K =\ZZ[i]$, the ring of integers of $K$.  Let
\begin{align*}
\mathcal{C}:=\left\{c\in \mathcal{O}_K : ~c\neq 1 \text{ is square-free and } c\equiv 1 \imod{\langle 16 \rangle} \right\}.
\end{align*}
In the same spirit as \eqref{Ld}, we define
\begin{equation*}
L_c(s)= L(s,\chi_c)L(s,\overline{\chi}_{c}),  \quad \mathcal{L}_{c}(s)=\log L_{c}(s),
\end{equation*}
  where $L(s,\chi_{c})$ is the Hecke $L$-function associated with the quartic residue symbol $\chi_{c}=\left(\frac{.}{c}\right)_4$.
Our result is
\begin{theorem}\label{mainthrm}
Let $\sigma>1/2$ and
\[\mathcal{S}(Y) = \#\left\{{c}\in \mathcal{C}: \norm(c)\leq Y \right\}. \]
Then there is a smooth density function $M_{\sigma}$ such that
\[\lim_{Y\to\infty}\frac{1}{\mathcal{S}(Y)}\#\left\{{c}\in \mathcal{C}: \norm(c)\leq Y\;\; \text{and}\;\;\ \mathcal{L}_{c}(\sigma) \leq z \right\}=\int\limits_{-\infty}^{z}M_{\sigma}(t)\; \dif t.\]
Furthermore, $M_\sigma$ can be constructed as the inverse Fourier transform of the characteristic function
 \begin{equation} \label{phiydef}
\varphi_{\sigma}(y)=\exp\left(-2iy\log(1-2^{-\sigma})\right)\prod_{\fp\nmid\langle2\rangle}
\left( \frac{1}{\norm(\fp)+1}+\frac{1}{4}\frac{\norm(\fp)}{\norm(\fp)+1}\sum_{j=0}^{3}\exp\left(-2iy\log\left|1-\frac{i^{j}}{\norm(\fp)^{\sigma}}\right|\right) \right).
\end{equation}
 \end{theorem}

Our proof of Theorem~\ref{mainthrm} closely follows the treatment of Theorem 1.3 in \cite{AH}.  We shall rely on the following two propositions for the proof of Theorem~\ref{mainthrm}.

\begin{prop}\label{mainprop1}
Set
\[\mathcal{S}^{*}(Y)=\sum_{c\in\mathcal{C}}\exp \left( -\norm(c)/Y \right).\]
Fix $\sigma>1/2$ and $y\in\mathbb{R}$.  We have
\begin{equation*}
\lim_{Y\to\infty}\frac{1}{\mathcal{S}^{*}(Y)} \sumSTAR_{{{c}}\in \mathcal{C}}\exp\left(iy\mathcal{L}_{{c}} (\sigma)\right)\exp \left( -\norm({c})/Y \right)=\varphi_{\sigma}(y),
\end{equation*}
where $\varphi_{\sigma}(y)$ is defined in \eqref{phiydef} and henceforth $\sum^{\star}$ indicates that the sum is over $c$ for which $L_c(\sigma)\neq 0$.
\end{prop}

\begin{prop}\label{mainprop2}
Let $\delta>0$ be given and $\sigma>1/2$ be fixed. For sufficiently large values of $y$, we have $$\left|\varphi_{\sigma}(y)\right|\leq\exp\left(-C|y|^{1/\sigma-\delta}\right),$$ where $C$ is a positive constant and can be chosen depending on values of $\delta$ and $\sigma$.
\end{prop}

 The above two propositions clearly have their cubic analogues in \cite{AH}.  Theorem~\ref{mainthrm} follows from the above two propositions using the exact same arguments as Section 2 of \cite{AH}.  Thus, we shall devote the remainder of the paper to the proofs of Propositions \ref{mainprop1} and \ref{mainprop2}, which are, following some preparatory work in Section~\ref{sec 2}, presented in Sections \ref{pfofProp1} and \ref{sec:good-density}, respectively.

\subsection{Notations} The following notations and conventions are used throughout the paper.\\
$f =O(g)$ or $f \ll g$ means $|f| \leq cg$ for some unspecified
positive constant $c$. \newline
$\epsilon$ denotes an arbitrary positive number, whereas $\epsilon_0$ denotes a fixed positive constant. \newline
$K=\QQ(i)$ and $\mathcal{O}_K$ denote the ring of integers of $K$.  \newline
The Gothic letters $\fa$, $\fb$, $\cdots$ represent ideals of $\mathcal{O}_K$. \newline
The norm of an integer $a \in \mathcal{O}_K$ is written as $\norm(a)$. The norm of an ideal $\fa$ is written as $\norm(\fa)$. \newline
$\mu_{[i]}$ denotes the M\"obius function on $\mathcal{O}_K$. \newline
$\zeta_{K}(s)$ is the Dedekind zeta function for the field $K$.

\section{Preliminaries}
\label{sec 2}

\subsection{Distribution and characteristic functions}
  A function $F: \mathbb{R} \rightarrow [0, 1]$ is said to be a distribution function if $F$ is non-decreasing, right-continuous with $F(-\infty)=0$ and
$F(+\infty)=1$.   For example,
$$F(z)=\int\limits_{-\infty}^{z} M(t) \dif t,$$
where $M(t)$ is non-negative and $\int_{-\infty}^{\infty} M(t) dt=1$. In this case, $M$ called tthe density function of $F$.
The characteristic function of $F$, $\varphi_{F}(y)$, is the Fourier transform of the measure $dF(z)$, i.e.
$$\varphi_F(y):= \int\limits_{-\infty}^{\infty} e^{iy z} \dif F(z).$$
For a more detailed discussion, we refer the reader to \cite{AH} and the references therein.

\subsection{Quartic residue symbol}
\label{sec2.0}
     The symbol $(\frac{\cdot}{n})_4$ is the quartic
residue symbol in the ring $\mz[i]$.  For a prime $\varpi \in \mz[i]$
with $\norm(\varpi) \neq 2$, the quartic character is defined for $a \in
\mz[i]$, $(a, \varpi)=1$ by $\leg{a}{\varpi}_4 \equiv
a^{(\norm(\varpi)-1)/4} \pmod{\varpi}$, with $\leg{a}{\varpi}_4 \in \{
\pm 1, \pm i \}$. When $\varpi | a$, we define
$\leg{a}{\varpi}_4 =0$.  Then the quartic character can be extended
to any composite $n$ with $(\norm(n), 2)=1$ multiplicatively. We extend the definition of $\leg{\cdot }{n}_4$ to $n=1$ by setting $\leg{\cdot}{1}_4=1$. \newline

   Note that in $\intz[i]$, every ideal co-prime to $2$ has a unique
generator congruent to 1 modulo $(1+i)^3$.  Such a generator is
called primary. Recall that \cite[Theorem 6.9]{Lemmermeyer} the quartic reciprocity law states
that for two primary integers  $m, n \in \mz[i]$,
\begin{align}
\label{quarticrec}
 \leg{m}{n}_4 = \leg{n}{m}_4(-1)^{((\norm(n)-1)/4)((\norm(m)-1)/4)}.
\end{align}

From the supplement theorem to the quartic reciprocity law (see for example, Lemma 8.2.1 and Theorem 8.2.4 in \cite{BEW}),
we have for $n=a+bi$ being primary,
\begin{align*}
  \leg {i}{n}_4=i^{(1-a)/2} \qquad \mbox{and} \qquad  \hspace{0.1in} \leg {1+i}{n}_4=i^{(a-b-1-b^2)/4}.
\end{align*}
   It follows that for any $c \equiv 1 \pmod {16}$,
\begin{align*}
  \chi_c(i)=\chi_c(1+i)=1.
\end{align*}
   The above shows that $\chi_c$ is trivial on units, hence it can be regarded as a primitive quartic character of the $\langle c \rangle$-ray class group of $K$ when $c$ is square-free.

\subsection{Evaluation of $\mathcal{S}^*(Y)$ and $\mathcal{S}(Y)$} \label{C}

To evaluate $\mathcal{S}^*(Y)$ and $\mathcal{S}(Y)$, defined in the statements of Proposition~\ref{mainprop1} and Theorem~\ref{mainthrm}, we note the following estimation from \cite[p. 7]{G&Zhao1}.
\begin{lemma}\label{lem:luo-lemma}
As $Y\to\infty$, we have for $(\fa, 2)=1$,
$$\sum_{\substack{c\in\mathcal{C}\\\gcd(\langle c \rangle),\fa)=1}}\exp\left(-\frac{\norm(c)}{Y}\right)=
C_{\fa}Y+O_\epsilon(Y^{1/2+\epsilon}\norm(\fa)^{\epsilon}),$$
where $$C_{\fa}=\frac{\res_{s=1}\zeta_{K}(s)}{\left| H_{\langle16\rangle} \right|\zeta_{K}(2)}\prod_{\substack{\fp|2\fa\\\fp\; \text{prime}}}
\left(1+\norm(\fp)^{-1}\right)^{-1},$$ and $\mathrm{res}_{s=1}\zeta_{K}(s)$ denotes the residue of $\zeta_{K}(s)$ at $s=1$, $H_{\langle 16\rangle}$ denotes the $\langle 16 \rangle$-ray class group of $K$.
\end{lemma}
  We deduce from Lemma \ref{lem:luo-lemma}, by setting $\fa=\langle1 \rangle$, that as $Y\rightarrow \infty$,
\begin{equation}
\label{N*}
\mathcal{S}^*(Y)= \sum_{\substack{c\in\mathcal{C}}}\exp\left(-\frac{\norm(c)}{Y}\right) \sim  \frac{2}{3}\frac{\res_{s=1}\zeta_{K}(s)}{\left| H_{\langle16\rangle} \right|\zeta_{K}(2)} Y,
\end{equation}
 and
 \[\mathcal{S}(Y)= \#\left\{c\in\mathcal{C}:\norm(c)\leq Y\right\}
\sim \frac{2}{3}\frac{\mathrm{res}_{s=1}\zeta_{K}(s)}{\left| H_{\langle 16\rangle}\right|\zeta_{K}(2)}Y.\]

\subsection{A zero density theorem}

   For $c\in \mathcal{C}$, the Hecke $L$-function associated with $\chi_c$ is defined by the Dirichlet series
$$L(s, \chi_c)=\sum_{0\neq \mathfrak{a} \subset \mathcal{O}_K } \frac{\chi_c(\mathfrak{a})}{\norm(\mathfrak{a})^s}, \; \Re(s)>1 .$$
$L(s, \chi_c)$ can be analytically continued to the entirety of $\comc$ and satisfies a functional equation relating its values at $s$ and at $1-s$.  We shall need the following zero density theorem for $L(s, \chi_c)$.
 \begin{lemma}{\cite[Corollary 1.6]{BGL}}
 \label{lem:zer-density}
 For $1/2< \sigma \leq  1$, $T\geq 1$ and $c\in\mathcal{C}$, let $N(\sigma,T, c)$
be the number of zeros $\rho =\beta +i\gamma$ of $L(s, \chi_{c})$ in the rectangle $\sigma\leq\beta\leq   1$ , $|\gamma|\leq T$ . Then
\[\sum_{\substack{c\in\mathcal{C}\\\norm(c)\leq Y}}N(\sigma,T,c)\ll Y^{g(\sigma)}T^{1+\frac{2-2\sigma}{3-2\sigma}}(YT)^{\varepsilon},\]
where \[g(\sigma)=\begin{cases}
\displaystyle \frac{8(1-\sigma)}{7-6\sigma},& \frac{1}{2}<\sigma\leq \frac56,\\ \\
\displaystyle \frac{2(10\sigma-7)(1-\sigma)}{24\sigma-12\sigma^2-11} & \frac56< \sigma\leq1.\end{cases}\]
 \end{lemma}

\subsection{The large sieve with quartic symbols and a P\'olya-Vinogradov type inequality}

  In the course of the proof of Theorem \ref{mainthrm}, we need the following large sieve type inequality for quartic residue symbols, which is a special case of \cite[Theorem 1.3]{BGL} and an improvement of \cite[Theorem 1.1]{G&Zhao}:
\begin{lemma}
\label{largesieve}
For $\varepsilon>0$ and $(b_\alpha)_{\alpha \in \mathcal{O}_K}$ be an arbitrary sequence of complex numbers, we have
\begin{equation}
\label{eqn:large-sieve}
\sumflat_{\substack{\lambda\equiv1\imod{\langle (1+i)^3 \rangle}\\\norm(\lambda)\leq M}}\left| \ \sumflat_{\substack{\alpha\equiv1\imod{\langle (1+i^3\rangle}\\
\norm(\alpha)\leq N}}b_{\alpha}~ \left(\frac{\alpha}{\lambda}\right)_4  \right|^2\ll_{\epsilon} \left( M+N+(MN)^{2/3} \right)(MN)^{\varepsilon}
\sumflat_{\norm(\alpha)\leq N}|b_{\alpha}|^{2},
\end{equation}
where $\sum^{\flat}$ means that the summation runs over the square-free elements of $\mathcal{O}_K$.
\end{lemma}

   We shall also need the following P\'olya-Vinogradov type inequality for $\mathfrak{f}$-ray class characters of $K$.
\begin{lemma}{\cite[Lemma 3.1]{G&Zhao}}
\label{H-P}
Let $K=\mathbb{Q}(i)$ and $\chi$ be a non-trivial character (not necessarily primitive) of $\mathfrak{f}$-ray class group of $K$. Then for $Y>1$ and $\varepsilon>0$, we have
\begin{equation}\label{eqn:polya}\sum_{a\equiv 1 \imod{\langle (1+i)^3 \rangle}}\chi(a)\exp \left( -\frac{\norm(a)}{Y} \right) \ll_{\epsilon}\norm(\mathfrak{f})^{1/2+\varepsilon}.
\end{equation}
\end{lemma}

\subsection{A Dirichlet series representation for $\exp\left(iy\mathcal{L}_{c} (s)\right)$}

   For $u \in \mathbb{C}$ and any non-negative integer $r$, we define the function $H_{r}(u)$ by $H_{0}(u)=1$ and for $r\geq1$,
\[H_{r}(u)=\frac{1}{r!}u(u+1)\cdots(u+r-1).\]
  This implies that
\begin{equation}
\label{H}
 \exp\left(-u\log(1-t)\right)=\sum_{r=0}^{\infty}H_{r}(u)t^{r},\quad\text{for } |t|<1.
\end{equation}
  We further define the arithmetic function $\lambda_{y}(\fa)$ on the integral ideals of $K$ as follows:
  \begin{equation} \label{lambdadef}
  \lambda_y(\fa)=\prod_{\mathfrak{p}} \lambda_y({\mathfrak{p}}^{\alpha_\mathfrak{p}})\quad\text{ and}\quad\lambda_y(\mathfrak{p} ^{\alpha_\mathfrak{p}} ) =H_{\alpha_\mathfrak{p}} \left(iy \right).
\end{equation}

   Similar to \cite[Lemma 4.1]{AH}, the following Lemma gives a Dirichlet series representation for $\exp\left(iy\mathcal{L}_{c} (s)\right)$.
\begin{lemma}\label{lem:exp(iyLc)}
Let $y\in\mathbb{R}$ and $s\in\mathbb{C}$ with $\Re(s)>1$. Then
\[
\exp\left(iy\mathcal{L}_{c}(s)\right)=\sum_{\substack{\fa,\fb\subset\mathcal{O}_K \\ \fa,\fb \neq 0}}\frac{\lambda_{y}(\fa)\lambda_{y}(\fb)\chi_{c}(\fa\fb^{3})}{\norm(\fa\fb)^{s}},
\]
where $\lambda_{y}$ is given in \eqref{lambdadef}.  Moreover, the above series is absolutely convergent.
\end{lemma}

   We omit the proof of Lemma \ref{lem:exp(iyLc)} as it is similar to the proof of \cite[Lemma 4.1]{AH}.  Moreover, It is shown in \cite[p. 92]{I-M} that for any $\varepsilon, R>0$ and all $|y|\leq R$,
we have
\begin{equation}\label{eqn:lambda-bound}
\lambda_{y}(\fa)\ll_{\epsilon,R}\norm(\fa)^{\varepsilon}.
\end{equation}

\section{PROOF OF PROPOSITION \ref{mainprop1}}
\label{pfofProp1}

To establish Proposition~\ref{mainprop1}, we first prove, in the next three section, Proposition~\ref{newprop} which gives a Dirichlet series (see \eqref{M}) representation for the limit in Proposition~\ref{mainprop1}.  Then in Section \ref{sec:prod}, we prove Proposition \ref{lem:euler} which renders a product representation for the afore-mentioned Dirichlet series to show that it is the same as $\varphi_{\sigma}(y)$ defined in \eqref{phiydef}. \newline

\begin{prop}
\label{newprop}
Fix $\sigma=1/2+\varepsilon_0$ for some
$\varepsilon_0>0$. Then for all $y\in\mathbb{R}$ we have
\begin{equation*}
\label{main-limit}
\lim_{Y\to\infty}\frac{1}{\mathcal{N}^{*}(Y)} \sumSTAR_{c\in \mathcal{C}}\exp\left(iy\mathcal{L}_{c} (\sigma)\right)\exp(-\norm(c)/Y)=\widetilde{M}_{\sigma}(y).
\end{equation*}
Here $\widetilde{M}_{\sigma}(y)$ is given by the following absolutely convergent Dirichlet series
\begin{align}
\label{M}
\widetilde{M}_{\sigma}(y)=\sum_{r_{1},r_{2}\geq0}\frac{\lambda_{y}(\langle1+i\rangle^{r_1})\lambda_{y}(\langle1+i\rangle^{r_2})}{2^{(r_1+r_{2})\sigma}}
\sum_{\substack{\fa,\fb,\fm\subset\mathcal{O}_K\\\gcd(\fa\fb\fm,\langle2\rangle)=1\\\gcd(\fa,\fb)=1}}
\frac{\lambda_{y}(\fa^4\fm)\lambda_{y}(\fb^4\fm)}{\norm(\fa^4\fb^4\fm^2)^{\sigma}\displaystyle{\prod_{\substack{\fp|\fa\fb\fm\\\fp\; \text{prime}}}
\left(1+\norm(\fp)^{-1}\right)}}.
\end{align}
\end{prop}

\subsection{Application of the zero density estimate}\label{sec:zero-density}
Let $A>0$ be fixed and $R_{Y, \varepsilon, A}$ be the rectangle with  the vertices $1\pm i(\log{Y})^A$ and $(1+\varepsilon)/2 \pm i(\log{Y})^A$.
Let $\mathcal{Z}^c$ be the set consisting of $c \in \mathcal{C}$ such that $L(s, \chi_c)$ does not vanish in $R_{Y, \epsilon, A}$.  Also, set $\mathcal{Z} = \mathcal{C} \setminus \mathcal{Z}^c$.  Note that $\mathcal{Z}$ and $\mathcal{Z}^c$ vary with $Y$, $\varepsilon$, and $A$. \newline

  Using arguments similar to those in the proof of \cite[Lemma 4.3]{AH}, we see that for $\sigma> 1/2$ as fixed in Proposition \ref{newprop} and
a sufficiently small $\varepsilon>0$, we have
\begin{equation}
\label{main1}
\sumSTAR_{c\in \mathcal{C}}\exp\left(iy\mathcal{L}_{c} (\sigma)\right)\exp(-\norm(c)/Y)= \sum_{c\in \mathcal{Z}^c} \exp\left(iy\mathcal{L}_{c} (\sigma)\right)\exp(-\norm(c)/Y)+ O(Y^\delta).
\end{equation}

Furthermore, for $c\in {\mathcal{Z}}^c$, $y\in\mathbb{R}$, and $1/2<\sigma\leq1$, we use Lemma \ref{lem:exp(iyLc)} and arguments similar to those in the proof of \cite[Lemma 4.4]{AH} to derive the following lemma giving a representation of $\exp\left(iy\mathcal{L}_{c} (\sigma)\right)$ as a sum of an infinite sum and a certain contour integral.
\begin{lemma}
\label{rep}
 Let $\varepsilon>0$ be given with $\sigma\geq 1/2+\varepsilon$. Suppose that $\sigma\leq1$.  If $c\in \mathcal{Z}^c$, then
$$\exp\left(iy\mathcal{L}_{c} (\sigma)\right)= \sum_{\substack{\fa,\fb\subset\mathcal{O}_K \\ \fa,\fb \neq 0}}\frac{\lambda_{y}(\fa)\lambda_{y}(\fb)\chi_{c}(\fa\fb^{3})}{\norm(\fa)^{{\sigma}}\norm(\fb)^{\sigma}}\exp\left(-\frac{\norm(\fa\fb)}{X}\right)-\frac{1}{2\pi i} \int\limits_{L_{Y, \epsilon, A}} \exp\left(iy\mathcal{L}_{c} (\sigma+u)\right)\Gamma(u) X^u \dif u, $$
where $L_{Y, \epsilon, A}$ is the contour that connects, by straight line segments, the points $(1-\sigma+\varepsilon/2) + i \infty$, $(1-\sigma+\varepsilon/2) + i (\log Y)^A$, $- \varepsilon/2 + i (\log Y)^A$, $-\varepsilon/2 - i(\log Y)^A$, $(1-\sigma+\varepsilon/2) - i (\log Y)^A$ and $(1-\sigma+\varepsilon/2) - i \infty$ .
\end{lemma}

Inserting the above lemma into \eqref{main1}, we get, for $\sigma\leq1$,
\begin{equation}
\label{main2}
\sumSTAR_{c\in \mathcal{C}}\exp\left(iy\mathcal{L}_{c} (\sigma)\right)\exp(-\norm(c)/Y)= (I)-(II)+(III)+ O(Y^\delta),
\end{equation}
where
\begin{equation}
\label{one}
(I)=\sum _{c\in \mathcal{C}} \left ( \sum_{\substack{\fa,\fb\subset\mathcal{O}_K \\ \fa,\fb \neq 0}}\frac{{\lambda}_{y}(\fa)\lambda_{y}(\fb)\chi_{c}(\fa\fb^{3})}{\norm(\fa)^{{\sigma}}\norm(\fb)^{\sigma}}\exp\left(-\frac{\norm(\fa\fb)}{X}\right)\right) \exp\left(-\frac{\norm(c)}{Y}\right),
\end{equation}
\begin{equation*}
(II)= \sum _{c\in \mathcal{Z}} \left ( \sum_{\substack{\fa,\fb\subset\mathcal{O}_K \\ \fa,\fb \neq 0}}\frac{{\lambda}_{y}(\fa)\lambda_{y}(\fb)\chi_{c}(\fa\fb^{3})}{\norm(\fa)^{{\sigma}}\norm(\fb)^{\sigma}}\exp\left(-\frac{\norm(\fa\fb)}{X}\right)\right) \exp\left(-\frac{\norm(c)}{Y}\right),
\end{equation*}
\begin{equation*}
(III)= \sum _{c\in \mathcal{Z}^c} \left ( -\frac{1}{2\pi i} \int\limits_{L_{Y, \epsilon, A}} \exp\left(iy\mathcal{L}_{c} (\sigma+u)\right)\Gamma(u) X^u \dif u
\right) \exp\left(-\frac{\norm(c)}{Y}\right).
\end{equation*}

   Letting $A>1$ and $X=Y^\eta$ for $\eta>0$, using arguments analogues to those in \cite{AH}, we deduce that
\begin{equation}
\label{two-estimate}
(II)+(III)\ll Y^\delta X^{1-\sigma+\epsilon}+Y X^{-\varepsilon/2} .
\end{equation}

\subsection{Evaluation of (I)}\label{sec:mean-value}

It still remains to prove an asymptotic formula for $(I)$ given in \eqref{one}.
\begin{lemma}\label{lem:luo-calcs}
Set
\begin{equation} \label{Cdef}
\widetilde{C}_{\sigma}(y)=\frac{2}{3}\frac{\res_{s=1}\zeta_{K}(s)}{\left| H_{\langle16\rangle} \right|\zeta_{K}(2)}\widetilde{M}_{\sigma}(y),
\end{equation}
with $\widetilde{M}_{\sigma}(y)$ defined in \eqref{M}.  Then
\begin{equation*}
(I)=\widetilde{C}_{\sigma}(y)Y+O\left( YX^{\varepsilon-\varepsilon_0}+Y^{1/2+\varepsilon}+Y^{1/2+2\epsilon}X^{1-\varepsilon_{0}+3\varepsilon} \right),
\end{equation*}
for any sufficiently small $\varepsilon>0$.
\end{lemma}
\begin{proof}
Starting with \eqref{one},
\begin{align*}(I)&=\sum_{\substack{\fa,\fb\subset\mathcal{O}_K \\ \fa,\fb \neq 0}}\frac{\lambda_{y}(\fa)\lambda_{y}(\fb)}{\norm(\fa\fb)^{\sigma}}\exp\left(-\frac{\norm(\fa\fb)}{X}\right)\sum_{c\in\mathcal{C}}\chi_{c}(\fa\fb^{3})\exp\left(-\frac{\norm(c)}{Y}\right)\\&
=\sum_{r_{1},r_{2}\geq0}\frac{\lambda_{y}(\langle1+i\rangle^{r_{1}})\lambda_{y}(\langle1+i\rangle^{r_2})}{2^{r_1\sigma+r_2\sigma}}\sum_{\substack{\fa,\fb\subset\mathcal{O}_K\\\gcd(\fa\fb,\langle2\rangle)=1}}\frac{\lambda_{y}(\fa)\lambda_{y}(\fb)}{\norm(\fa\fb)^{\sigma}}\exp\left(-\frac{2^{r_{1}+r_{2}}\norm(\fa\fb)}{X}\right) \sum_{c\in\mathcal{C}}\chi_{c}(\fa\fb^{3})\exp\left(-\frac{\norm(c)}{Y}\right).
\end{align*}
Rearranging the above sum over $\fa$ and $\fb$ with $\fm=\gcd(\fa,\fb)$ and recalling that $\chi_c(\fm^4)=1$ if $\fm$ is prime to $\langle c\rangle$, we obtain
\begin{equation} \label{I}
\begin{split}
(I)&=\sum_{r_{1},r_{2}\geq0}\frac{\lambda_{y}(\langle1+i \rangle^{r_1})\lambda_{y}(\langle1+i\rangle^{r_2})}{2^{r_1\sigma+r_2\sigma}}\sum_{\substack{\fa,\fb,\fm\subset\mathcal{O}_K\\\gcd(\fa\fb\fm,\langle2\rangle)=1\\\gcd(\fa,\fb)=1}}\frac{\lambda_{y}(\fa\fm)\lambda_{y}(\fb\fm)}{\norm(\fa\fb)^{\sigma}\norm(\fm)^{2\sigma}}\exp\left(-\frac{2^{r_{1}+r_{2}}\norm(\fa\fb\fm^2)}{X}\right)\\&
\hspace{4in}\times\sum_{\substack{c\in\mathcal{C}\\\gcd(\langle c\rangle,\fm)=1}}\chi_{c}(\fa\fb^{3})\exp\left(-\frac{\norm(c)}{Y}\right).
\end{split}
\end{equation}

The part of \eqref{I} contributed by the fourth powers is
\[ \sum_{r_{1},r_{2}\geq0}\frac{\lambda_{y}(\langle1+i\rangle^{r_1})\lambda_{y}(\langle1+i\rangle^{r_2})}{2^{r_1\sigma+r_2\sigma}}
\sum_{\substack{\fa,\fb,\fm\subset\mathcal{O}_K\\\gcd(\fa\fb\fm,\langle2\rangle)=1\\\gcd(\fa,\fb)=1}}\frac{\lambda_{y}(\fa^4\fm)\lambda_{y}(\fb^4\fm)}
{\norm(\fa\fb)^{4\sigma}\norm(\fm)^{2\sigma}}\exp\left(-\frac{2^{r_{1}+r_{2}}\norm(\fa^4\fb^4\fm^2)}{X}\right) \sum_{\substack{c\in\mathcal{C}\\\gcd(\langle c \rangle,\fa\fb\fm)=1}}\exp\left(-\frac{\norm(c)}{Y}\right). \]
Utilizing Lemma \ref{lem:luo-lemma} and the estimate \eqref{eqn:lambda-bound} for $\lambda_y$, above expression can be estimated by
\begin{align} \label{4powercontrib}
Y\sum_{r_{1},r_{2}\geq0}\frac{\lambda_{y}(\langle1+i\rangle^{r_1})\lambda_{y}(\langle1+i\rangle^{r_2})}{2^{r_1\sigma+r_2\sigma}}
\sum_{\substack{\fa,\fb,\fm\subset\mathcal{O}_K\\\gcd(\fa\fb\fm,\langle2\rangle)=1\\\gcd(\fa,\fb)=1}}
\frac{C_{\fa\fb\fm}\lambda_{y}(\fa^4\fm)\lambda_{y}(\fb^4\fm)}{\norm(\fa\fb)^{4\sigma}\norm(\fm)^{2\sigma}}\exp\left(-\frac{2^{r_{1}+r_{2}}
\norm(\fa^4\fb^4\fm^2)}{X}\right),
\end{align}
with an error that is $O\left( Y^{1/2+\varepsilon} \right)$.  Here $C_{\fa\fb\fm}$ is defined in Lemma \ref{lem:luo-lemma}.  Choose $\varepsilon$ small enough so that $0<\varepsilon < \varepsilon_0$. Inserting the formula
$$\exp\left(-\frac{2^{r_{1}+r_{2}}\norm(\fa^4\fb^4\fm^2)}{X}\right)=\frac 1{2\pi i} \int\limits_{(1)}\Gamma(u)\left(\frac{2^{r_{1}+r_{2}}\norm(\fa^4\fb^4\fm^2)}{X}\right)^{-u}\; \dif u$$
into \eqref{4powercontrib} and shifting the line of integration to $\Re(u)=\varepsilon-\varepsilon_0$, we conclude the contribution of fourth powers to $(I)$ is
\begin{align}\label{llast}Y\sum_{r_{1},r_{2}\geq0}\frac{\lambda_{y}(\langle1+i\rangle^{r_1})\lambda_{y}(\langle1+i\rangle^{r_2})}{2^{r_1\sigma+r_2\sigma}}\sum_{\substack{\fa,\fb,\fm\subset\mathcal{O}_K\\\gcd(\fa\fb\fm,\langle2\rangle)=1\\\gcd(\fa,\fb)=1}}\frac{C_{\fa\fb\fm}\lambda_{y}(\fa^4\fm)\lambda_{y}(\fb^4\fm)}{\norm(\fa\fb)^{4\sigma}\norm(\fm)^{2\sigma}}+O\left( YX^{\epsilon-\epsilon_0} + Y^{1/2+\varepsilon} \right).\end{align}
Mark that the coefficient of $Y$ in \eqref{llast} agrees with $\widetilde{C}_\sigma(y)$ given by \eqref{Cdef}. \newline

To estimate the contribution of non-fourth powers to $(I)$,
we write $\fa=\langle a\rangle$ and $\fb=\langle b\rangle$ with $a,\;b\in\ZZ[i]$ being primary for ideals $\fa$ and $\fb$ with $\gcd(\fa\fb,\langle2\rangle)=1$.
We then have, by the quartic reciprocity law \eqref{quarticrec},
\begin{equation}
\label{preRS}
\begin{split}
\sum_{c\in \mathcal{C}}\chi_{c}(\fa\fb^3)\exp\left(-\frac{\norm(c)}{Y}\right)
&=\sum_{c\equiv1\imod{\langle 16 \rangle}}\chi_{c}(\fa\fb^3)\exp\left(-\frac{\norm(c)}{Y}\right)
\sum_{\substack{d^{2}|c\\d\equiv1\imod{\langle(1+i)^3\rangle}}}\mu_{[i]}(d) \\
&=\sum_{c\equiv1\imod{\langle16\rangle}}\chi_{ab^3}(c)\exp\left(-\frac{\norm(c)}{Y}\right)
\sum_{\substack{d^{2}|c\\d\equiv1\imod{\langle (1+i)^3\rangle}}}\mu_{[i]}(d),
\end{split}
\end{equation}
where we use $\mu_{[i]}$, the M\"obius function in $\intz[i]$, to detect the square-free condition on $c$. Here $\mu_{[i]}(d)$ for any $d \in \mz[i]$ is defined to be $1$ if the ideal generalized by $d$ equals $\mz[i]$ and to be $(-1)^r$ if the ideal generalized by $d$ equals a product of $r$ distinct prime ideals. For other values of $d$, $\mu_{[i]}(d)$ is defined to be $0$. \newline

Now we split the last expression in \eqref{preRS} into two parts to get
\begin{align*}\sum_{c\in \mathcal{C}}\chi_{c}(\fa\fb^3)\exp\left(-\frac{\norm(c)}{Y}\right)
&=R+S,
\end{align*}
  where
\begin{align*}
R &=\sum_{\substack{d\equiv1\imod{\langle (1+i)^3 \rangle}\\\norm(d)\leq B}}\mu_{[i]}(d)\chi_{ab^3}(d^2)\sum_{c\equiv\overline{d}^2\imod{\langle 16 \rangle}}
\chi_{ab^3}(c)\exp\left(-\frac{\norm(d^2c)}{Y}\right), \\
S &=\sum_{d_1\equiv1\imod{\langle(1+i)^3\rangle}}\chi_{ab^3}(d_{1}^2)\sum_{\substack{d|d_1\\d\equiv1\imod{\langle(1+i)^3\rangle}\\\norm(d)>B}}\mu_{[i]}(d)
\sumflat_{c\equiv\overline{d}_1^2\imod{\langle16\rangle}}\chi_{ab^3}(c)\exp\left(-\frac{\norm(d_1^2c)}{Y}\right).
\end{align*}
Here $B$ is a parameter to be optimized later, $\overline{d}$ (respectively $\overline{d}_1$) is the multiplicative inverse of $d$ (respectively $d_1$) modulo $\langle 16 \rangle$ and $\sum^{\flat}$ denotes summation over square-free elements of $\mathcal{O}_K$. \newline

We have
\begin{align*}
R&=\sum_{\substack{d\equiv1\imod{\langle(1+i)^3\rangle}\\\norm(d)\leq B}}\mu_{[i]}(d)\chi_{ab^3}(d^2)\sum_{c\equiv\overline{d}^2\imod{\langle16\rangle}}
\chi_{ab^3}(c)\exp\left(-\frac{\norm(d^2c)}{Y}\right)\\
&=\frac{1}{\left|H_{\langle16\rangle}\right|}\sum_{\substack{d\equiv1\imod{\langle(1+i)^3\rangle}\\\norm(d)\leq B}}\mu_{[i]}(d)\chi_{ab^3}(d^2)
\sum_{c\equiv1\imod{\langle(1+i)^3\rangle}}\chi_{ab^3}(c)\sum_{\chi\imod{\langle16\rangle}}\chi(d^2c)\exp\left(-\frac{\norm(d^2c)}{Y}\right)\\
&=\frac{1}{\left|H_{\langle16\rangle}\right|}\sum_{\chi\imod{\langle16\rangle}}\sum_{\substack{d\equiv1\imod{\langle(1+i)^3 \rangle}\\\norm(d)\leq B}}\mu_{[i]}(d)\chi_{ab^3}\chi(d^2)\sum_{c\equiv1\imod{\langle(1+i)^3\rangle}}\chi_{ab^3}\chi(c)\exp\left(-\frac{\norm(d^2c)}{Y}\right).
\end{align*}
The bound in \eqref{eqn:polya} gives
\begin{equation*}
R\ll B\norm(ab^3)^{1/2+\varepsilon}.
\end{equation*}
Therefore, the summands in \eqref{I} involving $R$ can be majorized by
\begin{align*}
&\ll B\sum_{r_{1},r_{2}}\frac{\left|\lambda_{y}(\langle1+i\rangle^{r_{1}})\right|\left|\lambda_{y}(\langle1+i\rangle^{r_{2}})\right|}{2^{r_{1}\sigma+r_{2}\sigma}}
\sum_{\fa,\fb,\fm}\frac{\norm(\fa\fb^{3})^{1/2+\varepsilon}\left|\lambda_{y}(\fa\fm)\right|\left|\lambda_{y}(\fb\fm)\right|}{\norm(\fa)^{\sigma}
\norm(\fb)^{\sigma}\norm(\fm)^{2\sigma}}\exp\left(-\frac{2^{r_{1}+r_{2}}\norm(\fa\fb\fm^{2})}{X}\right)\\
&\ll B\sum_{\fa,\fb}\norm(\fa)^{-\sigma+1/2+2\varepsilon}\norm(\fb)^{-\sigma+3/2+4\varepsilon}\exp\left(-\frac{\norm(\fa\fb)}{X}\right)
\end{align*}
With a change of variables, the last expression is recast as
\[  B\sum_{\fa}\norm(\fa)^{-\sigma+1/2+2\varepsilon}\sum_{\fb|\fa}\norm(\fb)^{1+2\varepsilon}\exp\left(-\frac{\norm(\fa)}{X}\right) \ll B\sum_{\fa}\norm(\fa)^{-\sigma+3/2+5\varepsilon}\exp\left(-\frac{\norm(\fa)}{X}\right)\ll BX^{5/2-\sigma+6\varepsilon}. \]

Next, the summands \eqref{I} with $S$ are precisely
\begin{align*}
&\sum_{r_{1},r_{2}\geq 0}\frac{\lambda_{y}(\langle1+i\rangle^{r_{1}})\lambda_{y}(\langle1+i\rangle^{r_{2}})}{2^{r_{1}\sigma+r_{2}\sigma}}
\sum_{\substack{a\equiv1\imod{\langle(1+i)^3\rangle}\\b\equiv1\imod{\langle(1+i)^3\rangle}\\m\equiv1\imod{\langle(1+i)^3\rangle}
\\(\langle a\rangle,\langle b\rangle)=1}}\frac{\lambda_{y}(\langle am\rangle)\lambda_{y}(\langle bm\rangle)}{\norm(a)^{\sigma}\norm(b)^{\sigma}
\norm(m)^{2\sigma}}\exp\left(-\frac{2^{r_{1}+r_{2}}\norm(abm^2)}{X}\right)\\
&\hspace{1in}\times\sum_{d_1\equiv1\imod{\langle(1+i)^3\rangle}}\chi_{ab^3}(d_{1}^2)\sum_{\substack{d|d_1\\d\equiv1\imod{\langle(1+i)^3\rangle}\\
\norm(d)>B}}\mu_{[i]}(d)\sumflat_{c\equiv\overline{d_1}^2\imod{\langle16\rangle}}\chi_{ab^3}(c)\exp\left(-\frac{\norm(d_1^2c)}{Y}\right).\nonumber
\end{align*}
Now we rewrite $a$ as $a=a_{1}a_{2}^{2}$, where $a_{1}$ is the square-free part of $a$.   Also note that we may assume that
\[ \norm(a_1)\ll \frac{X^{1+\varepsilon}}{\norm(a_2^2b)^{1+\varepsilon}}, \; \norm(c)\ll\frac{Y^{1+\varepsilon}}{\norm(d_{1}^2)^{1+\varepsilon}} \]
and $B<\norm(d_1)<\sqrt{Y}$.  Proceeding in a manner similar to the treatment of $S$ in \cite{AH} by using Cauchy-Schwarz inequality and the large sieve inequality \eqref{eqn:large-sieve}, we see that the contribution of $S$ to \eqref{I} is
\begin{equation*}
\ll (XY)^{1/2+2\varepsilon}+Y^{1+2\varepsilon}\frac{X^{1-\sigma+\varepsilon}}{B}+Y^{5/6+3\varepsilon/2}\frac{X^{1-\sigma+\varepsilon}}{B^{2/3}}.
\end{equation*}
The combined contribution of $R$ and $S$ to $(I)$ is
\begin{equation*}
\ll BX^{5/2-\sigma+6\varepsilon}+(XY)^{1/2+2\varepsilon}+Y^{1+2\varepsilon}\frac{X^{1-\sigma+\varepsilon}}{B}+Y^{5/6+3\varepsilon/2}\frac{X^{1-\sigma+\varepsilon}}{B^{2/3}}.
\end{equation*}
Upon taking $B=Y^{1/2+\varepsilon}X^{-3/4-5\varepsilon/2}$, the above is
\begin{equation*}
\ll Y^{1/2+2\varepsilon}X^{7/4-\sigma+4\varepsilon}.
\end{equation*}
Combining this estimation with \eqref{llast}, we obtain

\begin{equation}\label{one-estimate}
(I)=\widetilde{C}_{\sigma}(y)Y+O(YX^{\varepsilon-\varepsilon_0}+Y^{1/2+\varepsilon}+Y^{1/2+2\varepsilon}X^{5/4-\varepsilon_{0}+4\varepsilon})
\end{equation}
and complete the proof of the lemma.
\end{proof}

\subsection{Proof of Proposition \ref{newprop}}
\begin{proof}
  As the treatment for $\sigma>1$ is analogue to the one given in the proof of \cite[Proposition 4.2]{AH}, we may assume that $\sigma\leq1$.
Inserting \eqref{two-estimate} and \eqref{one-estimate} into \eqref{main2} yields

\begin{equation*}
\sumSTAR_{c\in \mathcal{C}}\exp\left(iy\mathcal{L}_{c} (\sigma)\right)\exp(-\norm(c)/Y)=\widetilde{C}_{\sigma}(y) Y+O\left(YX^{\epsilon-\epsilon_0}+Y^{\frac{1}{2}+\epsilon}+Y^{1/2+2\varepsilon}X^{5/4-\varepsilon_{0}+4\varepsilon} +Y^\delta X^{1/2-\varepsilon_0+\varepsilon}  + Y X^{-\varepsilon/2}\right).
\end{equation*}
Now choosing $X=Y^\eta$ for a sufficiently small positive constant $\eta$ and using \eqref{N*} give the result.
\end{proof}

\subsection{The product formula for
$\widetilde{M}_{\sigma}(y)$}
\label{sec:prod}

\begin{prop}\label{lem:euler}
  Let $\widetilde{M}_{\sigma}(y)$ be given in \eqref{M}. We have
$\widetilde{M}_{\sigma}(y)=\varphi_{\sigma}(y)$, where $\varphi_{\sigma}(y)$ is defined in Theorem \ref{mainthrm}.
\end{prop}
\begin{proof}
Using \eqref{H} yeilds
\begin{equation} \label{sumoverr}
\begin{split}
\sum_{r_{1},r_{2}\geq0}\frac{\lambda_{y}(\langle1+i\rangle^{r_1})\lambda_{y}(\langle1+i\rangle^{r_2})}{2^{(r_1+r_{2})\sigma}} & =\sum_{r_{1}\geq0}
\frac{\lambda_{y}(\langle1+i\rangle^{r_1})}{2^{r_1\sigma}}\sum_{r_{2}\geq0}
\frac{\lambda_{y}(\langle1+i\rangle^{r_2})}{2^{r_2\sigma}} \\
& =\sum_{r_{1}\geq0}\frac{H_{r_1}(iy)}{2^{r_1\sigma}}\sum_{r_{2}\geq0}\frac{H_{r_2}(iy)}{2^{r_2\sigma}} =\exp\left(-2iy\log(1-2^{-\sigma})\right).
\end{split}
\end{equation}

In the sequal, let $\fp$ denote a prime ideal and we adopt the convention that all products over $\fp$ are restricted to odd prime ideals, i.e. prime ideals co-prime to $\langle2\rangle$. Set
\begin{equation} \label{N1}
\begin{split}
\widetilde{N}_{\sigma}(y)&:=\sum_{\substack{\fa,\fb,\fm\subset\mathcal{O}_K\\\gcd(\fa\fb\fm,\langle2\rangle)=1\\\gcd(\fa,\fb)=1}}\frac{\lambda_{y}(\fa^4\fm)
\lambda_{y}(\fb^4\fm)}{\norm(\fa\fb)^{4\sigma}\norm(\fm)^{2\sigma}}\displaystyle{\prod_{\substack{\fp|\fa\fb\fm}}\left(1+\norm(\fp)^{-1}\right)^{-1}} \\
&= \sideset{}{^{\sharp}}\sum_{\fm} \frac{1}{\norm(\fm)^{2\sigma}}\prod_{\fp|\fm}(1+\norm(\fp)^{-1})^{-1}\sum_{\fa}\frac{\lambda_{y}
(\fa^{4}\fm)}{\norm(\fa)^{4\sigma}}\prod_{\substack{\fp|\fa\\\fp\nmid\fm}}(1+\norm(\fp)^{-1})^{-1}
\sum_{\fb}\frac{\lambda_{y}(\fb^{4}\fm)}{\norm(\fb)^{4\sigma}}\prod_{\substack{\fp|\fb\\\fp\nmid\fa\fm}}(1+\norm(\fp)^{-1})^{-1}.
\end{split}
\end{equation}
Here $\sum^{\sharp}$ denotes that the sum runs over $\fm$'s that are free of fourth powers. \newline

We need to have an Euler product for $\widetilde{N}_{\sigma}(y)$.  To this end, we first find an Euler product for the innermost sum over $\fb$ in the last expression of \eqref{N1}. Let $\nu_\fp(\fm)$ denote the multiplicity of a prime ideal $\fp$ in an ideal $\fm$, i.e. the highest power of $\fp$ that divides $\fm$.  We have
\begin{equation} \label{N2}
\sum_{\fb}\frac{\lambda_{y}(\fb^{4}\fm)}{\norm(\fb)^{4\sigma}} \prod_{\substack{\fp|\fb\\\fp\nmid\fa\fm}}(1+\norm(\fp)^{-1})^{-1} =\prod_{\fp} \mathcal{F}(\fp)   \frac{\displaystyle \prod_{\fp|\fa\fm}\left(\sum_{j=0}^{\infty}\frac{\lambda_{y}(\fp^{4j+\nu_{\fp}(\fm)})}{\norm(\fp)^{4j\sigma}}\right)}
{\displaystyle \prod_{\fp|\fa\fm} \mathcal{F}(\fp) } ,
\end{equation}
where
\[ \mathcal{F}(\fp) = 1+\sum_{j=1}^{\infty}\frac{\lambda_{y}(\fp^{4j})}{\norm(\fp)^{4j\sigma}}\left(1+\norm(\fp)^{-1}\right)^{-1} . \]

Inserting \eqref{N2} into \eqref{N1}, we arrive at
\begin{align}
\label{N3}
\widetilde{N}_{\sigma}(y)=\prod_{\fp} \mathcal{F} (\fp) \widetilde{P}_{\sigma}(y), \quad \mbox{where} \quad \widetilde{P}_{\sigma}(y) = \sideset{}{^{\sharp}}\sum_{\fm}\frac{1}{\norm(\fm)^{2\sigma}}\prod_{\fp|\fm}
\frac{ \mathcal{G}(\fp, \nu_{\fp}(\fm)) }
{\mathcal{F}(\fp) } \sum_{\fa}\frac{\lambda_{y}(\fa^{4}\fm)}{\norm(\fa)^{4\sigma}}\prod_{\substack{\fp|\fa\\\fp\nmid\fm}}
\frac{\mathcal{G}(\fp, \nu_{\fp}(\fm))}
{\mathcal{F}(\fp)} ,
\end{align}
where
\[ \mathcal{G}(\fp, l) = \sum_{j=0}^{\infty}\frac{\lambda_{y}(\fp^{4j+l})}{\norm(\fp)^{4j\sigma}}\left(1+\norm(\fp)^{-1}\right)^{-1} . \]
Now we have
\begin{equation} \label{N4}
\begin{split}
\widetilde{P}_{\sigma}(y) & =\sideset{}{^{\sharp}}\sum_{\fm} \frac{1}{\norm(\fm)^{2\sigma}} \prod_{\fp|\fm} \frac{\mathcal{G}(\fp, \nu_{\fp}(\fm))}{\mathcal{F}(\fp) } \prod_{\fp\nmid\fm} \left(1+\sum_{k=1}^{\infty} \frac{\lambda_{y}(\fp^{4k})}
{\norm(\fp)^{4k\sigma}} \frac{\mathcal{G}(\fp, 0)}{ \mathcal{F}(\fp)} \right)  \prod_{\fp|\fm}
\left(\sum_{k=0}^{\infty}\frac{\lambda_{y}(\fp^{4k+v_{\fp}(\fm)})}{\norm(\fp)^{4k\sigma}}\right)  \\
&=\prod_{\fp}\left(1+\sum_{k=1}^{\infty}\frac{\lambda_{y}(\fp^{4k})}{\norm(\fp)^{4k\sigma}}
\frac{ \mathcal{G}(\fp, 0) }{\mathcal{F}(\fp) } \right) \sideset{}{^{\sharp}}\sum_{\fm} \frac{1}{\norm(\fm)^{2\sigma}}\frac{\displaystyle\prod_{\fp|\fm}
\frac{\mathcal{G}(\fp, \nu_{\fp}(\fm))}{\mathcal{F}(\fp)}
\displaystyle  \prod_{\fp|\fm}\left(\displaystyle\sum_{k=0}^{\infty}\frac{\lambda_{y}(\fp^{4k+v_{\fp}(\fm)})}{\norm(\fp)^{4k\sigma}}\right)}
{\displaystyle \prod_{\fp|\fm}\left(\displaystyle 1+ \sum_{k=1}^{\infty}\frac{\lambda_{y}(\fp^{4k})}{\norm(\fp)^{4k\sigma}}
\frac{\displaystyle \mathcal{G}(\fp, 0)}{\mathcal{F}(\fp)} \right)}.
\end{split}
\end{equation}

The summation over $\fm$ in the above expression can be recast as the Euler product
\begin{align}
\label{N5}
\prod_{\fp}\left(1+\sum_{l=1}^{3}\frac{1}{\norm(\fp)^{2l\sigma}}
\frac{\frac{\displaystyle \mathcal{G}(\fp, l)}
{\displaystyle  \mathcal{F}(\fp) } \displaystyle \sum_{k=0}^{\infty}\frac{\lambda_{y}(\fp^{4k+l})}{\norm(\fp)^{4k\sigma}}}
{1+\displaystyle \sum_{k=1}^{\infty}\frac{\lambda_{y}(\fp^{4k})}{\norm(\fp)^{4k\sigma}} \frac{\displaystyle \mathcal{G}(\fp, 0)}
{\displaystyle  \mathcal{F}(\fp) }} \right).
\end{align}

Inserting \eqref{N5} into \eqref{N4} and using resulting expression for $\widetilde{P}_{\sigma}(y)$ in \eqref{N3}, we infer that
\begin{align*}
\widetilde{N}_{\sigma}(y)=\prod_{\fp}\widetilde{M}_{\sigma,\fp}(y), \quad
\widetilde{M}_{\sigma,\fp}(y)=1-\left(1+\norm(\fp)^{-1}\right)^{-1}
+\left(1+\norm(\fp)^{-1}\right)^{-1}\sum_{l=0}^{3}\sum_{\substack{j=0\\k=0}}^{\infty}\frac{\lambda_{y}(\fp^{4j+l})\lambda_{y}(\fp^{4k+l})}{\norm(\fp)^{(4j+l)\sigma}\norm(\fp)^{(4k+l)\sigma}}.
\end{align*}

The definition of $\lambda_{y}$ in \eqref{lambdadef}, together with the relation
\begin{align}
\label{ortho}
\sum_{l=0}^{3}i^{lk}=\begin{cases}
4 & \text{ if }k \equiv 0 \pmod {4}, \\
0 & \text{ otherwise},
 \end{cases}
\end{align}
gives that
\begin{align*}
\sum_{l=0}^{3}\sum_{\substack{j=0\\k=0}}^{\infty}\frac{\lambda_{y}(\fp^{4j+l})\lambda_{y}(\fp^{4k+l})}{\norm(\fp)^{(4j+l)\sigma}\norm(\fp)^{(4k+l)\sigma}}
&=\sum_{l=0}^{3}\sum_{j=0}^{\infty}\frac{H_{4j+l}\left(iy\right)}{\norm(\fp)^{(4j+l)\sigma}}
\sum_{k=0}^{\infty}\frac{H_{4k+l}\left(iy\right)}{\norm(\fp)^{(4k+l)\sigma}}, \\
&=\frac 1{16}\sum_{l=0}^{3}\sum_{r=0}^{\infty}\frac{H_{r}\left(iy\right)}{\norm(\fp)^{r\sigma}}
\sum^{3}_{n=0}i^{(r-l)n}\sum_{r'=0}^{\infty}\frac{H_{r'}\left(iy\right)}{\norm(\fp)^{r'\sigma}}\sum^{3}_{m=0}i^{(r'-l)m}  \nonumber \\
&=\frac1{16}\sum_{l=0}^{3}\sum_{\substack{n=0\\m=0}}^{3}\frac{1}{i^{l(n+m)}}\exp\left(-iy \left(\log \left(1-\frac{i^{n}}
{\norm(\fp)^{\sigma}}\right )+\log \left(1-\frac{i^{m}}{\norm(\fp)^{\sigma}}\right)\right) \right ), \nonumber
\end{align*}
  where the last expression above follows from \eqref{H}. \newline

Further using the relation \eqref{ortho} for $k=m+n$,
we conclude that
\begin{align*}
\sum_{l=0}^{3}\sum_{\substack{j=0\\k=0}}^{\infty}\frac{\lambda_{y}(\fp^{4j+l})\lambda_{y}(\fp^{4k+l})}{\norm(\fp)^{(4j+l)\sigma}\norm(\fp)^{(4k+l)\sigma}}&=\frac14\sum_{j=0}^{3}\exp\left(-2iy\log\left|1-\frac{i^{j}}{\norm(\fp)^{\sigma}}\right|\right).
\end{align*}
Therefore, $\widetilde{N}_{\sigma}(y)$ takes the form
\begin{align}
\label{NM}
\widetilde{N}_{\sigma}(y)=\prod_{\fp \nmid\langle2\rangle}\widetilde{M}_{\sigma,\fp}(y), \quad \widetilde{M}_{\sigma,\fp}(y)=\left( \frac{1}{\norm(\fp)+1}+\frac{1}{4}\left(\frac{\norm(\fp)}{\norm(\fp)+1}\right)\sum_{j=0}^{3}\exp\left(-2iy\log\left|1-\frac{i^{j}}{\norm(\fp)^{\sigma}}\right|\right)\right).
\end{align}

  The assertion of Proposition \ref{lem:euler} now follows from this and \eqref{sumoverr}.
\end{proof}

\section{PROOF OF PROPOSITION  \ref{mainprop2}}
\label{sec:good-density}

Since $\overline{\varphi}_{\sigma}(y)=\varphi_{\sigma}(-y)$, we can assume, without loss of generality, that $y>0$.  It follows from Proposition \ref{lem:euler} that
\begin{equation*}
\varphi_{\sigma}(y)=\exp\left(-2iy\log(1-2^{-\sigma})\right)\prod_{\fp\nmid\langle2\rangle}\widetilde{M}_{\sigma,\fp}(y),
\end{equation*}
where $\widetilde{M}_{\sigma,\fp}(y)$ is given in \eqref{NM}. Note first that for all $y$ we have $|\widetilde{M}_{\sigma,\fp}(y)|\leq 1$. Further note that
\begin{align*}
\widetilde{Q}_{\sigma,\fp}(y)&:=\sum_{j=0}^{3}\exp\left(-2iy\log\left|1-\frac{i^{j}}{\norm(\fp)^{\sigma}}\right|\right).
\\&=\exp\left(-2iy\log(1-N(p)^{-\sigma})\right)\left(1+2\exp\left(-2iy\log\frac{\sqrt{\norm(\fp)^{2\sigma}+1}}{\norm(\fp)^{\sigma}-1}\right)+
\exp\left(-2iy\log\frac{\norm(\fp)^{\sigma}+1}{\norm(\fp)^{\sigma}-1}\right)\right).
\end{align*}
 Thus,
\begin{align*}
\left|\widetilde{Q}_{\sigma,\fp}(y)\right|&=\left|1+2\exp\left(-2iy\log\frac{\sqrt{\norm(\fp)^{2\sigma}+1}}{\norm(\fp)^{\sigma}-1}\right)+\exp\left(-2iy\log\frac{\norm(\fp)^{\sigma}+1}{\norm(\fp)^{\sigma}-1}\right)\right|\\
& \leq \left|1+2\exp\left(-2iy\log\frac{\sqrt{\norm(\fp)^{2\sigma}+1}}{\norm(\fp)^{\sigma}-1}\right)\right|+1 =\sqrt{1+8\cos^{2}\left(y\log\frac{\sqrt{\norm(\fp)^{2\sigma}+1}}{\norm(\fp)^{\sigma}-1}\right)}+1.
\end{align*}
Given any $\varepsilon>0$ and sufficiently large $y$, consider the prime ideals $\fp$ with
\begin{equation}\label{eqn:condition2}
1.35-\varepsilon\leq y\log\frac{\sqrt{\norm(\fp)^{2\sigma}+1}}{\norm(\fp)^{\sigma}-1}\leq 1.77+\varepsilon .
\end{equation}
Upon taking $\varepsilon$ small enough, we can ensure that
\[ \left|\cos\left(y\log\frac{\sqrt{\norm(\fp)^{2\sigma}+1}}{\norm(\fp)^{\sigma}-1}\right) \right|\leq 0.22 . \]
This implies that $\left|\widetilde{Q}_{\sigma,\fp}(y)\right|\leq 2.2$.
It follows that for all $\fp$ satisfying (\ref{eqn:condition2}), we have
\[\left|\widetilde{M}_{\sigma,\fp}(y)\right|\leq \frac{1}{\norm(\fp)+1}+0.55\left(\frac{\norm(\fp)}{\norm(\fp)+1}\right)\leq 0.8.\]
Observe that
$\frac{2y}{3.54}\leq \norm(\fp)^{\sigma}\leq \frac{2y}{2.7}$ is equivalent to
 \[ \frac{2.7}{2}\norm(\fp)^{\sigma}\log\tfrac{\sqrt{\norm(\fp)^{2\sigma}+1}}{\norm(\fp)^{\sigma}-1}\leq y\log\tfrac{\sqrt{\norm(\fp)^{2\sigma}+1}}{\norm(\fp)^{\sigma}-1}\leq \frac{3.54}{2}\norm(\fp)^{\sigma}\log\tfrac{\sqrt{\norm(\fp)^{2\sigma}+1}}{\norm(\fp)^{\sigma}-1}.\]
 Since
 $$\displaystyle{\lim_{\norm(\fp)\to\infty}\norm(\fp)^{\sigma}\log\tfrac{\sqrt{\norm(\fp)^{2\sigma}+1}}{\norm(\fp)^{\sigma}-1}}=1,$$
 we get that for sufficiently large $y$,
\[\frac{2y}{3.54}\leq \norm(\fp)^{\sigma}\leq \frac{2y}{2.7} \]
which implies
\[ 1.35-\varepsilon\leq y\log\frac{\sqrt{\norm(\fp)^{2\sigma}+1}}{\norm(\fp)^{\sigma}-1}\leq 1.77 +\varepsilon.\]
Let $\Pi(x)$ be the number of prime ideals in $K$ with norms not exceeding $x$, and let $\Pi_{\sigma}(y)$ be the number of prime ideals satisfying \eqref{eqn:condition2}.   It follows from the above consideration that if $y$ is large enough, then
\[\Pi_{\sigma}(y)>\Pi\left(\left(\frac{2y}{2.7}\right)^{1/\sigma}\right)-\Pi\left(\left(\frac{2y}{3.54}\right)^{1/\sigma}\right) \gg_{\sigma} y^{1/\sigma-\delta}\]
for all sufficiently small $\delta>0$.  Consequently,
\[\left|\varphi_{\sigma}(y)\right|=\prod_{\fp\nmid\langle2\rangle}\left|\widetilde{M}_{\sigma,\fp}(y)\right|\leq 0.8^{\Pi_{\sigma}(y)}=\exp\left(-\log\left(\frac{1}{0.8}\right)\Pi_{\sigma}(y)\right)\leq\exp\left(-Cy^{1/\sigma-\delta}\right),\]
where $C$ can be chosen according to the values of $\sigma$ and $\delta$.  This completes the proof of Proposition \ref{mainprop2}.  \newline

\noindent{\bf Acknowledgments.} P. G. is supported in part by NSFC grant 11871082 and L. Z. by the FRG grant PS43707 and the Faculty Silverstar Award PS49334.  Parts of this work were done when P. G. visited the University of New South Wales (UNSW) in August 2018. He wishes to thank UNSW for the invitation, financial support and warm hospitality during his pleasant stay.  Finally, the authors would like to thank the anonymous referee for his/her comments and suggestions.

\bibliography{biblio}
\bibliographystyle{amsxport}

\vspace*{.5cm}

\noindent\begin{tabular}{p{8cm}p{8cm}}
School of Mathematical Sciences & School of Mathematics and Statistics \\
Beihang University & University of New South Wales \\
Beijing 100191 China & Sydney NSW 2052 Australia \\
Email: {\tt penggao@buaa.edu.cn} & Email: {\tt l.zhao@unsw.edu.au} \\
\end{tabular}

\end{document}